\documentclass[12pt,leqno]{article} 
\usepackage{amsmath,amssymb,bbm}
\usepackage{amsthm}
\usepackage{titlesec} 

\titlespacing{\section}{0cm}{3.5pc}{1.5pc}

\makeatletter
\def\@citex[#1]#2{\if@filesw\immediate\write\@auxout{\string\citation{#2}}\fi
  \def\@citea{}\@cite{\@for\@citeb:=#2\do
    {\@citea\def\@citea{\@citesep}\@ifundefined
       {b@\@citeb}{{\bf ?}\@warning
       {Citation `\@citeb' on page \thepage \space undefined}}%
{\csname b@\@citeb\endcsname}}}{#1}}
\def\@citesep{; }
\makeatother

\newtheoremstyle{Kang}{}{}{\itshape}{}{\bf}{}{.5em}{}
\theoremstyle{Kang}
\newtheorem{theorem}{Theorem}[section]

\newtheoremstyle{Kremark}{}{}{}{}{\bf}{}{.5em}{}
\theoremstyle{Kremark}
\newtheorem*{remark}{Remark.}
\newtheorem{defn}[theorem]{Definition}

\newtheorem{other}{}

\newenvironment{Case}[1]{\medskip {\it Case #1.}}{}

\allowdisplaybreaks[1]  

\def\fn#1{\operatorname{#1}} 
\def\bm#1{\mathbbm{#1}}

\title{Rational Invariants for Subgroups of $S_5$ and $S_7$}
\author{\begin{minipage}{0.37\textwidth}
Ming-chang Kang \\[2mm] \normalsize
Department of Mathematics \\
National Taiwan University \\
Taipei, Taiwan \\
E-mail: kang@math.ntu.edu.tw
\end{minipage}
and \
\begin{minipage}{0.5\textwidth}
Baoshan Wang \\[2mm] \normalsize
School of Mathematics and System Sciences \\
Beihang University \\
Beijing, China \\
E-mail: bwang@buaa.edu.cn
\end{minipage} }
\date{}

\begin{document}

\maketitle

\footnote{\textit{\!\!\!$2010$ Mathematics Subject
Classification}. 13A50, 14E08.} \footnote{\textit{\!\!\!Keywords
and phrases}. Noether's problem, rationality problem, rational
invariants.} \footnote{\textit{\!\!\!This paper was finished while
the second-named author was visiting National Taiwan University.}}

\begin{abstract}
{\noindent\bf Abstract.} Let $G$ be a subgroup of $S_n$, the
symmetric group of degree $n$. For any field $k$, $G$ acts
naturally on the rational function field $k(x_1,x_2,\ldots,x_n)$
via $k$-automorphisms defined by $\sigma\cdot x_i=x_{\sigma(i)}$
for any $\sigma\in G$, any $1\le i\le n$. Theorem. If $n\le 5$,
then the fixed field $k(x_1,\ldots,x_n)^G$ is purely
transcendental over $k$. We will show that
$\bm{C}(x_1,\ldots,x_7)^G$ is also purely transcendental over
$\bm{C}$ if $G$ is any transitive subgroups of $S_7$ other than
$A_7$; a similar result is valid for solvable transitive subgroups
of $S_{11}$.
\end{abstract}

\newpage
\section{Introduction}

Let $k$ be a field.
A finitely generated field extension $L$ of $k$ is called $k$-rational if $L$ is purely transcendental over $k$;
it is called stably $k$-rational if $L(x_1,x_2,\ldots,x_m)$ is $k$-rational where $x_1,\ldots,x_m$ are elements which are algebraically independent over $L$.

Let $G$ be a subgroup of $S_n$ where $S_n$ is the symmetric group
of degree $n$. For any field $k$, $G$ acts naturally on the
rational function field $k(x_1,\ldots,x_n)$ via $k$-automorphisms
defined by $\sigma\cdot x_i=x_{\sigma(i)}$ for any $\sigma \in G$,
any $1\le i\le n$. Noether's problem asks whether the fixed field
$k(x_1,\ldots,x_n)^G:=\{f\in k(x_1,\ldots,x_n):\sigma(f)=f$ for
all $\sigma\in G\}$ is $k$-rational (resp.\ stably $k$-rational)
\cite{No}. If $G$ is embedded in $S_N$ through the left regular
representation (where $N=|G|$), it is easy to see that
$k(x_1,\ldots,x_N)^G$ is $k$-isomorphic to $k(V_{\fn{reg}})^G$
where $\rho:G\to GL(V_{\fn{reg}})$ is the regular representation
of $G$, i.e.\ $V_{\fn{reg}}=\bigoplus_{g\in G} k\cdot e_g$ is a
$k$-vector space and $h\cdot e_g=e_{hg}$ for any $h, g\in G$. We
will write $k(G)=k(V_{\fn{reg}})^G$ in the sequel. The rationality
problem of $k(G)$ is also called Noether's problem, e.g.\ in the
paper of Lenstra \cite{Le}.

Noether's problem is related to the inverse Galois problem,
to the existence of generic $G$-Galois extensions,
and to the existence of versal $G$-torsors over $k$-rational field extensions.
For a survey of this problem, see \cite{GMS,Sa,Sw}.

We will recall some previous results for the rationality problem of $k(x_1,\ldots,x_n)^G$ where $G$ is a subgroup of $S_n$
and $\sigma\cdot x_i=x_{\sigma(i)}$ for any $\sigma\in G$, any $1\le i\le n$.
If $G_1$ is another subgroup of $S_n$ which is conjugate to $G$ within $S_n$,
it is easy to see that $k(x_1,\ldots,x_n)^G$ is $k$-rational if and only if so is $k(x_1,\ldots,x_n)^{G_1}$ over $k$.
Thus it suffices to consider only one group $G$ in each conjugacy class of subgroups of $S_n$.
The case when $n\le 3$ is easy;
the answer is affirmative.

When $n=4$ and $G=A_4$ the alternating group, the rationality
problem of $\bm{C}(x_1,x_2,x_3,x_4)^{A_4}$ was studied by W.\
Burnside \cite{Bu}. For an intriguing account of this situation,
see the article \cite{CHK}. When $n=4$ and $k=\bm{C}$, the problem
was solved completely by J.\ A.\ Tyrrell and C.\ M.\ Williams
\cite{TW}.

\begin{theorem}[Tyrrell and Williams \cite{TW}] \label{t1.1}
Let $G$ be any subgroup of $S_4$.
Then $\bm{C}(x_1,x_2,x_3,x_4)^G$ is $\bm{C}$-rational.
\end{theorem}

A result related to Theorem \ref{t1.1} was solved by Kitayama, Yamasaki \cite{KY}, Kang and Zhou \cite{KZ}.

\begin{theorem}[\cite{KY,KZ}] \label{t1.2}
Let $G$ be a finite subgroup of $GL_4(\bm{Q})$.
Let $G$ act on $\bm{Q}(x_1,x_2,x_3,x_4)$ by $\bm{Q}$-automorphisms defined by $\sigma\cdot x_j=\sum_{1\le i\le 4} a_{ij} x_i$
where $\sigma=(a_{ij})_{1\le i,j\le 4}\in G\subset GL_4(\bm{Q})$.
Then $\bm{Q}(x_1,x_2,x_3,x_4)^G$ is $\bm{Q}$-rational if and only if $G$ is not conjugate to the image
of a faithful representation of $C_8$ or $C_3\rtimes C_8$ in $GL_4(\bm{Q})$.
\end{theorem}

When $k$ is any field and $\sigma=(a_{ij})_{1\le i,j\le 4} \in GL_4(k)$,
a necessary and sufficient condition for $k(x_1,x_2,x_3,x_4)^{\langle \sigma \rangle}$ to be $k$-rational was given in \cite{Ka}.

We will prove the following result.

\begin{theorem} \label{t1.3}
Let $k$ be any field, $G$ be a subgroup of $S_n$. Let $G$ act on
the rational function field $k(x_1,\ldots,x_n)$ via
$k$-automorphisms defined by $\sigma\cdot x_i=x_{\sigma(i)}$ for
any $\sigma\in G$, any $1\le i\le n$. If $n=4$ or 5, then
$k(x_1,\ldots,x_n)^G$ is $k$-rational.
\end{theorem}

In the case when $p=7$ or $11$, two related results are obtained.

\begin{theorem} \label{t1.4}
Let $k$ be any field, $G$ be a transitive subgroup of $S_7$. Let
$G$ act on the rational function field $k(x_1,\ldots,x_7)$ via
$k$-automorphisms defined by $\sigma\cdot x_i=x_{\sigma(i)}$ for
any $\sigma\in G$, any $1\le i\le 7$. If $G$ is not isomorphic to
the group $PSL_2(\bm{F}_7)$ or the group $A_7$, then
$k(x_1,\ldots,x_7)^G$ is $k$-rational.

Moreover, when $G$ is isomorphic to $PSL_2(\bm{F}_7)$ and $k$ is a
field satisfying that $\fn{char}k =0$ and $\sqrt{-7} \in k$, then
$k(x_1,\ldots,x_7)^G$ is also $k$-rational.

\end{theorem}

\begin{theorem} \label{t1.5}
Let $k$ be any field, $G$ be a transitive solvable subgroup of
$S_{11}$. Let $G$ act on the rational function field
$k(x_1,\ldots,x_{11})$ via $k$-automorphisms defined by
$\sigma\cdot x_i=x_{\sigma(i)}$ for any $\sigma\in G$, any $1\le
i\le {11}$. Then $k(x_1,\ldots,x_{11})^G$ is $k$-rational.

\end{theorem}

We will emphasize that we choose to prove $k(x_1,\ldots,x_n)^G$ is
$k$-rational for any field $k$ in Theorem \ref{t1.3}, Theorem
\ref{t1.4} and Theorem \ref{t1.5}. The corresponding results when
$k=\bm{Q}$ or the situation that $k(G)$ is $k$-rational are
special cases or consequences of the above three theorems. In the
literature, some authors dealt with only the case $\fn{char}k=0$.
We will find that the proof of Theorem \ref{t1.3} when
$\fn{char}k=5$ or 2 requires extra efforts also; see Theorem
\ref{t3.2} and the proof of Case 5 in Section 3.

We remark that, when $n\ge 6$, it is still unknown whether $k(x_1,
\cdots, x_n)^{A_n}$ is $k$-rational or not; the answer is unknown
even when $k=\bm{C}$.

One may consider monomial representations in the above Theorem
\ref{t1.3}, instead of permutation representations. We point out
that a necessary and sufficient condition for
$k(x_1,x_2,x_3,x_4)^{\langle \sigma \rangle}$ to be $k$-rational
where $\sigma : x_1 \to x_2 \to x_3 \to x_4 \to -x_1$ is given in
\cite[Theorem 1.8]{Ka}. In particular, $\bm{Q}
(x_1,x_2,x_3,x_4)^{\langle \sigma \rangle}$ is not stably
$\bm{Q}$-rational. The situation for monomial representations of
dimension $4$ follows easily from previous results of Yamasaki
\cite{Ya} and Hoshi-Kitayama-Yamasaki \cite{HKY} on the
$3$-dimensional monomial actions if char $k \neq 2$. However, the
case when char $k = 2$ requires further investigation. In order to
solve the rationality problem for monomial representations of
dimension $5$, it is conceivable that many challenging questions
will arise.

The proof of Theorem \ref{t1.3} will be given in Section 3 (see
Theorem \ref{t3.3} and Theorem \ref{t3.4}). The proof of Theorem
\ref{t1.4} and Theorem \ref{t1.5} is given in Section 4. Theorem
\ref{t4.7} is of interest itself. The rationality problem of fixed
fields by subgroups of $S_6$ will be discussed in a separate
article.

\bigskip
Standing terminology. Throughout the paper, we will denote by
$S_n, A_n, C_n, D_n$ the symmetric group of degree $n$, the
alternating group of degree $n$, the cyclic group of order $n$,
and the dihedral group of order $2n$ respectively. If $k$ is any
field, $k(x_1,\ldots,x_n)$ denotes the rational function field of
$n$ variables over $k$. When $\rho:G\to GL(V)$ is a representation
of $G$ over a field $k$, then $k(V)$ denotes the rational function
field $k(x_1,\ldots,x_n)$ with the induced action of $G$ where
$\{x_1,\ldots,x_n\}$ is a basis of the dual space $V^*$ of $V$. In
particular, when $V=V_{reg}$ is the regular representation space,
denote by $\{x(g) : g \in G \}$ a dual basis of $V_{reg}$; then
$k(V_{reg})=k(x(g) : g \in G)$ where $h \cdot x(g) =x(hg)$ for any
$h,g \in G$. We will write $k(G):=k(V_{reg})^G$.

\section{Preliminaries}

In this section we recall some known results which will be applied to solve the rationality problem in Theorem \ref{t1.3}.

\begin{theorem} \label{t2.2}
Let $G$ be a finite group acting on $L(x_1,\ldots,x_m)$, the
rational function field of $m$ variables over a field $L$. Assume
that (i) for any $\sigma\in G$, $\sigma(L)\subset L$, and (ii) the
restriction of the action of $G$ to $L$ is faithful.

{\rm (1) (\cite[Theorem 1]{HK3})}
Assume furthermore that, for any $\sigma\in G$,
\[
\begin{pmatrix} \sigma(x_1) \\ \sigma(x_2) \\ \vdots \\ \sigma(x_m) \end{pmatrix}
= A(\sigma) \cdot \begin{pmatrix} x_1 \\ x_2 \\ \vdots \\ x_m \end{pmatrix}+B(\sigma)
\]
where $A(\sigma)\in GL_m(L)$ and $B(\sigma)$ is an $m\times 1$
matrix over $L$. Then there exist $z_1,\ldots,z_m \in
L(x_1,\ldots,x_m)$ so that $L(x_1,\ldots,x_m)=L(z_1,\ldots,z_m)$
with $\sigma(z_i)=z_i$ for any $\sigma\in G$, any $1\le i\le m$.

In fact, there are $(a_{ij})_{1 \le i,j \le n} \in GL_n(L)$ and
$c_j \in L$ such that, for $1 \le j \le n$, $z_j = \sum_{1\le i
\le n} a_{ij} x_i +c_j$. Moreover, if $B(\sigma)=0$ for all
$\sigma \in G$, then we may choose $z_j$ simply by $z_j =
\sum_{1\le i \le n} a_{ij} x_i$.

{\rm (2) (\cite[Theorem $1'$]{HK3})}
Assume furthermore that, for any $\sigma\in G$,
\[
\begin{pmatrix} \sigma(x_1) \\ \sigma(x_2) \\ \vdots \\ \sigma(x_m) \end{pmatrix}
=A(\sigma) \begin{pmatrix} x_1 \\ x_2 \\ \vdots \\ x_m \end{pmatrix}
\]
where $A(\sigma) \in GL_m(L)$ and $G$ acts on $L(x_1/x_m,x_2/x_m,
\ldots, x_{m-1}/x_m)$ naturally. Then there exist $z_1,\ldots,z_m
\in L(x_1,\ldots,x_m)$ so that $L(x_1/x_m,
\ldots,x_{m-1}/x_m)=L(z_1/z_m,z_2/z_m$, $\ldots,z_{m-1}/z_m)$ and
$\sigma(z_i/z_m)=z_i/z_m$ for any $\sigma\in G$, any $1\le i\le
m-1$.
\end{theorem}

\begin{theorem}[{\cite[Theorem 3.1]{AHK}}] \label{t2.3}
Let $L$ be a field, $L(x)$ be the rational function field of one
variable over $L$ and $G$ be a finite group acting on $L(x)$.
Suppose that, for any $\sigma\in G$, $\sigma(L)\subset L$ and
$\sigma(x)=a_\sigma x+b_\sigma$ where $a_\sigma, b_\sigma \in L$
and $a_\sigma\ne 0$. Then $L(x)^G=L^G(f)$ for some polynomial
$f\in L[x]$. In fact, if $m=\min \{\deg g(x): g(x)\in
L[x]^G\backslash L\}$, any polynomial $f\in L[x]^G$ with $\deg
f=m$ satisfies the property $L(x)^G=L^G(f)$.
\end{theorem}

\begin{theorem}[{\cite[Lemma 2.7]{HK2}}] \label{t2.8}
Let $k$ be any field, $a,b\in k\backslash \{0\}$ and $\sigma:
k(x_1,x_2)\to k(x_1,x_2)$ be a $k$-automorphism defined by
$\sigma(x_1)=a/x_1$, $\sigma(x_2)=b/x_2$. Then
$k(x_1,x_2)^{\langle \sigma \rangle}\allowbreak =k(u,v)$ where
\[
u=\frac{x_1-\frac{a}{x_1}}{x_1x_2-\frac{ab}{x_1x_2}}, \quad
v=\frac{x_2-\frac{b}{x_2}}{x_1x_2-\frac{ab}{x_1x_2}}.
\]
\end{theorem}

\begin{defn} \label{d2.5}
Let $\sigma$ be a $k$-automorphism on the rational function field
$k(x_1,\ldots,x_n)$. $\sigma$ is called a purely monomial
automorphism if $\sigma(x_j)=\prod_{1\le i\le n} x_i^{a_{ij}}$ for
$1\le j\le n$ where $(a_{ij})_{1\le i,j\le n} \in GL_n(\bm{Z})$.
The action of a finite group $G$ acting on $k(x_1,\ldots,x_n)$ is
called a purely monomial action if, for all $\sigma\in G$,
$\sigma$ acts on $k(x_1,\ldots,x_n)$ by a purely monomial
$k$-automorphism \cite{HK1}.
\end{defn}

\begin{theorem}[\cite{HK1,HK2,HR}] \label{t2.6}
Let $k$ be any field, $G$ be a finite group acting on the rational function field $k(x_1,x_2,x_3)$ by purely monomial $k$-automorphisms.
Then the fixed field $k(x_1,x_2,x_3)^G$ is $k$-rational.
\end{theorem}

\begin{theorem}[Maeda \cite{Ma}] \label{t2.4}
Let $k$ be any field, $A_5$ be the alternating group of degree $5$
acting on $k(x_1,\ldots,x_5)$. Let $A_5$ act on
$k(x_1,\ldots,x_5)$ via $k$-automorphisms defined by $\sigma\cdot
x_i=x_{\sigma(i)}$ for any $\sigma\in A_5$, any $1\le i\le 5$.
Then both the fixed fields
$k(x_1/x_5,x_2/x_5,x_3/x_5,x_4/x_5)^{A_5}$ and
$k(x_1,x_2,x_3,x_4,x_5)^{A_5}$ are $k$-rational.
\end{theorem}

\begin{theorem}[Kemper \cite{Ke}] \label{t2.9}
Let $k$ be any field satisfying that char$k =0$ and $\sqrt{-7} \in
k$, $G$ be the group $PSL_2(\bm{F}_7)$. Then there is a faithful
representation $G \to GL(V)$ such that dim$_k V =3$ and $k(V)^G$
is $k$-rational.

\end{theorem}

\medskip Recall the definition of $k(G)$ at the end of Section 1.
The following theorem is a special case of Noether's problem,
which was investigated by many people \cite{Sw}. For a proof, see
\cite[Corollary 7.3]{Le}.

\begin{theorem} \label{t2.7}
Let $k$ be any field. If $n\le 46$ and $8\nmid n$, then $k(C_n)$
is $k$-rational.
\end{theorem}

\section{Subgroups of $S_5$}

\begin{defn} \label{d3.1}
Let $p$ be a prime number, $G:=(\bm{Z}/p\bm{Z})\rtimes (\bm{Z}/p\bm{Z})^\times$.
We will present $G$ as a permutation subgroup of $S_p$ as follows.
Let $\bar{a}\in \bm{Z}/p\bm{Z}$ be a primitive root modulo $p$, i.e.\ $(\bm{Z}/p\bm{Z})^\times =\langle \bar{a} \rangle$.
Define $\sigma: x_i\mapsto x_{i+1}$,
$\tau: x_i\mapsto x_{ai}$ where $0\le i\le p-1$ and the indices of $x_i$ are taken modulo $p$.
By identifying $\sigma$ and $\tau$ as elements of $\bm{Z}/p\bm{Z}$ and $(\bm{Z}/p\bm{Z})^\times$,
it is clear that $G=\langle \sigma,\tau\rangle$ with relations $\sigma^p=\tau^{p-1}=1$ and $\tau\sigma\tau^{-1}=\sigma^a$.
Thus $G\subset S_p$.

For any positive integer $d$ with $d\mid p-1$, write $p-1=de$.
Denote by $G_{pd}$ the group $G_{pd}=\langle \sigma,\tau^e \rangle \subset G$.
\end{defn}

It is known that a transitive solvable subgroup of $S_p$ is
conjugate to a subgroup of $G_{p(p-1)}$ \cite[p.117, Proposition
11.6; DM, p.91, Exercise 3.5.1]{Co}. For the classification of
transitive non-solvable subgroups of $S_p$, see \cite[p. 99]{DM}.
As a consequence of the classification of finite simple groups,
the groups $S_n, A_n$ and the Mathieu groups are the only
$4$-transitive permutation groups \cite[p. 34]{DM}.

\begin{theorem} \label{t3.2}
Let $k$ be a field with $\fn{char} k=p>0$, $k(x_i: 0\le i\le p-1)$ be the rational function field of $p$ variables.
Let $G_{pd}=\langle \sigma,\tau^e\rangle$ be the group in Definition \ref{d3.1} where $p-1=de$.
Let $G_{pd}$ act on $k(x_i: 0\le i\le p-1)$ via $k$-automorphisms defined by $\sigma: x_i \mapsto x_{i+1}$,
$\tau^e: x_i \mapsto x_{a^e i}$ where $\bar{a}\in \bm{Z}/p\bm{Z}$ is a primitive root modulo $p$ and the indices of $x_i$ are taken modulo $p$.
Then $k(x_i:0\le i\le p-1)^{G_{pd}}$ is $k$-rational.
\end{theorem}

\begin{proof}
Write $G=G_{pd}$.

By Theorem \ref{t2.3}, $k(x_i:0\le i\le p-1)^G=k(x_i/x_0:1\le i\le p-1)^G(t)$ for some element $t$ with $\lambda(t)=t$ for all $\lambda \in G$.

On the other hand, note that $\bar{a}\in \bm{Z}/p\bm{Z} \simeq \bm{F}_p \subset k$.
Consider the action of $G$ on the rational function field $k(y_1,y_2)$ defined by
\begin{align*}
\sigma &: y_1 \mapsto y_1+y_2, ~ y_2 \mapsto y_2, \\
\tau &: y_1 \mapsto \bar{a}^{-e} y_1, ~ y_2\mapsto y_2.
\end{align*}

Define $u=y_1/y_2 \in k(y_1,y_2)$. Then $\sigma(u)=u+1$,
$\tau(u)=\bar{a}^{-e} u$.

Clearly $G$ acts faithfully on $k(y_1,y_2)$ and $k(u)$.

Since $G$ acts faithfully on $k(x_i/x_0:1\le i\le p-1)$, we may
apply Part (2) of Theorem \ref{t2.2}. It follows that
$k(x_i/x_0,u:1\le i\le p-1)=k(x_i/x_0,s:1\le i\le p-1)$ for some
element $s$ with $\lambda(s)=s$ for all $\lambda \in G$.

Hence $k(x_i: 0\le i\le p-1)^G=k(x_i/x_0:1\le i\le p-1)^G (t) \simeq k(x_i/x_0:1\le i\le p-1)^G(s)=k(x_i/x_0,u:1\le i\le p-1)^G$.

Since $G$ acts faithfully on $k(u)$, apply Part (2) of Theorem
\ref{t2.2} to $k(x_i/x_0, u:1\le i\le p-1)$ with $L=k(u)$. We get
$k(x_i/x_0,u: 1\le i\le p-1)^G=k(u)^G(v_i:1\le i\le p-1)$ where
$\lambda(v_i)=v_i$ for all $1\le i\le p-1$, for all $\lambda \in
G$.

Note that $k(u)^G$ is $k$-rational by L\"uroth's Theorem.
It follows that $k(x_i:0\le i\le p-1)^G$ is $k$-rational.
\end{proof}

\begin{remark}
By applying Theorems 1.1 of \cite{KP}, it is possible to prove the
stable rationality of $k(x_i:0\le i\le p-1)^{G_{pd}}$ in the above
theorem; but it is seems difficult to prove the rationality of it,
without the device of the above theorem.

On the other hand, the stable rationality of $\bm{Q}(x_i:0\le i\le
p-1)^{G_{pd}}$ will be discussed in a separate article. When
$p=5$, see Theorem \ref{t3.4}; when $p=7$, see the next section.

\end{remark}

The following result is a generalization of Theorem \ref{t1.1}.

\begin{theorem} \label{t3.3}
Let $k$ be any field.
Let $G$ be a subgroup of $S_n$ acting on the rational function field $k(x_1,\ldots,x_n)$ by $\sigma\cdot x_i=x_{\sigma(i)}$ for any $\sigma\in G$,
any $1\le i\le n$.
If $n\le 4$, then $k(x_1,\ldots,x_n)^G$ is $k$-rational.
\end{theorem}

\begin{proof}
The case $n=2$ or 3 is easy.
For example, when $G=\langle (1~2~3)\rangle \subset S_3$,
the rationality of $k(x_1,x_2,x_3)^G$ can be shown by applying Theorem \ref{t2.7}.

From now on we consider the case $n=4$.
By Theorem \ref{t2.3}, $k(x_1,x_2,x_3,x_4)^G=k(x_1/x_4,x_2/x_4,x_3/x_4)^G(t)$ for some $t$ with $\sigma(t)=t$ for all $\sigma\in G$.
Since $G$ acts on $k(x_1/x_4,x_2/x_4,x_3/x_4)^G$ by purely monomial $k$-automorphisms,
it follows that $k(x_1/x_4,\allowbreak x_2/x_4,x_3/x_4)^G$ is $k$-rational by Theorem \ref{t2.6}.
\end{proof}

\begin{theorem} \label{t3.4}
Let $k$ be any field, $G$ be any subgroup of $S_5$.
If $G$ acts on the rational function field $k(x_1,x_2,x_3,x_4,x_5)$ by $\sigma\cdot x_i=x_{\sigma(i)}$ for any $\sigma\in G$, any $1\le i\le 5$,
then $k(x_1,x_2,x_3,x_4,x_5)^G$ is $k$-rational.
\end{theorem}

\begin{proof}
First of all note that, if $G$ is not a transitive subgroup of
$S_5$, then the question is reduced to Theorem \ref{t3.3}. For
example, suppose that there are two $G$-orbits, $\{x_1,x_2,x_3\}$
and $\{x_4,x_5\}$. Let $G_1$ be the restriction of $G$ to
$k(x_1,x_2,x_3)$ (i.e.\ $G_1$ is the image of $G$ in $\fn{Aut}_k
k(x_1,x_2,x_3)$), and $G_2$ be the restriction of $G$ to
$k(x_4,x_5)$. Then $G_2=\{1\}$ or $S_2$. If $2\mid |G_1|$, then
$G_1$ acts faithfully on $k(x_1,x_2,x_3)$ and
$k(x_1,\ldots,x_5)^G=k(x_1,x_2,x_3)^G(t_1,t_2)$ for some $t_1$,
$t_2$ with $\lambda(t_1)=t_1$, $\lambda(t_2)=t_2$ for all
$\lambda\in G$ by applying Part (1) of Theorem \ref{t2.2}. If
$2\nmid |G_1|$, i.e.\ $|G_1|=3$, then $G\simeq G_1\times G_2$.
When $|G_1|=3$ and $|G_2|=2$, then $G=\langle \sigma\rangle$ where
$\sigma=(1,2,3)(4,5)$. Hence
$k(x_1,\ldots,x_5)^G=\{k(x_1,x_2,\ldots,x_5)^{\langle
\sigma^3\rangle}\}^{\langle
\sigma\rangle}=k(x_1,x_2,x_3)(x_4+x_5,x_4x_5)^{\langle\sigma\rangle}
=k(x_1,x_2,x_3)^{\langle\sigma\rangle}(x_4+x_5,x_4x_5)$. Hence the
result.

From now on, we will assume that $G$ is a transitive subgroup of $S_5$.

As mentioned before, it suffices to show that
$k(x_1,\ldots,x_5)^G$ is $k$-rational where $G$ is a transitive
subgroup of $S_5$ in each conjugacy class of subgroups in $S_5$.

There are only 5 such conjugacy classes.
We choose a representative in each class.
We get
\[
S_5,~ A_5, ~ G_{20},~ D_5,~ C_5
\]
where $C_5$ is the cyclic group of order $5$, $G_{20}$ is a group
of order 20 and is exactly the group $G_{p(p-1)}$ in Definition
\ref{d3.1} with $p=5$. Note that the dihedral group $D_5$ is the
group $G_{5\cdot 2}$ in Definition \ref{d3.1}.

\begin{Case}{1} $G=S_5$. \end{Case}

The rationality of $k(x_1,\ldots,x_5)^{S_5}$ is easy.

\begin{Case}{2} $G=A_5$. \end{Case}

The rationality of $k(x_1,\ldots,x_5)^{A_5}$ follows from Theorem \ref{t2.4}.

\begin{Case}{3} $G=C_5$. \end{Case}

The rationality of $k(x_1,\ldots,x_5)^{C_5}$ follows from Theorem \ref{t2.7}.

\begin{Case}{4} $G=G_{20}$. \end{Case}

If $\fn{char}k=5$, the rationality of $k(x_1,\ldots,x_5)^G$ follows from Theorem \ref{t3.2}.

From now on, we may assume $\fn{char}k\ne 5$. We want to show that
$k(x_i:0\le i\le 4)^G$ is $k$-rational where $G=G_{20}$.

Recall that $G=\langle \sigma,\tau\rangle$ where $\sigma: x_i\mapsto x_{i+1}$, $\tau:x_i\mapsto x_{2i}$ for $0\le i\le 4$.

Write $\zeta=\zeta_5$ where $\zeta_5$ is a primitive 5th-root of unity.
$\pi=\fn{Gal}(k(\zeta)/k)=\langle \lambda \rangle$.
Then $\pi \simeq C_4$, $C_2$ or $\{1\}$.

\medskip
\begin{Case}{4.1} $\pi=\{1\}$, i.e.\ $\zeta\in k$. \end{Case}

For $0\le i\le 4$, define
\begin{equation}
y_i=\sum_{0\le j\le 4} \zeta^{-ij} x_j. \label{eq3.1}
\end{equation}

Then $\sigma(y_i)=\zeta^i y_i$, $\tau(y_i)=y_{3i}$ for $0\le i\le 4$.

It follows that $k(x_i:0\le i\le 4)^{\langle \sigma\rangle}=k(y_i:0\le i\le 4)^{\langle\sigma\rangle}=k(z_i:0\le i\le 4)$
where $z_0=y_0$, $z_1=y_1^5$, $z_i=y_i/y_1^i$ for $2\le i\le 4$.

Note that
\[
\tau: z_0\mapsto z_0,~ z_1\mapsto z_1^3 z_3^5, ~ z_2\mapsto
1/(z_1z_3^2),~ z_3\mapsto z_4/(z_1z_3^3),~ z_4\mapsto z_2/(z_1^2
z_3^4).
\]

Define $u_1=z_2/z_3$, $u_i=\tau^{i-1}(u_1)$ for $2\le i\le 4$.
Then we get
\[
\tau: u_1\mapsto u_2\mapsto u_3\mapsto u_4\mapsto u_1,
\]
because $u_2=z_3/z_4$, $u_3=z_1z_3z_4/z_2$, $u_4=z_1z_2z_4$.

Since $k(z_i:1\le i\le 4)=k(u_i:1\le i\le 4)$,
it follows that $k(x_i:0\le i\le 4)^G=k(z_i:0\le i\le 4)^{\langle \tau\rangle}=k(u_i:1\le i\le 4)^{\langle\tau\rangle}(z_0)$ is $k$-rational by Theorem \ref{t2.7}.

\medskip
\begin{Case}{4.2} $\pi\simeq C_4$.
We may assume that $\pi=\langle\lambda\rangle$ with $\lambda(\zeta)=\zeta^2$. \end{Case}

Extend the action of $G$ on $k(x_i:0\le i\le 4)$ to the action of $\langle G,\lambda \rangle$ on $k(\zeta)(x_i:0\le i\le 4)$
by requiring $\sigma(\zeta)=\tau(\zeta)=\zeta$ and $\lambda(x_i)=x_i$ for $0\le i\le 4$.
It follows that $k(x_i: 0\le i\le 4)^G=\{k(\zeta)(x_i:0\le i\le 4)^{\langle\lambda\rangle}\}^G=k(\zeta)(x_i:0\le i\le 4)^{\langle\sigma,\tau,\lambda\rangle}$.

Define $y_i$ by the same formula as \eqref{eq3.1}.
Then we have
\begin{align*}
\sigma &: \zeta\mapsto \zeta,~ y_i\mapsto \zeta^i y_i, \\
\tau &: \zeta\mapsto \zeta,~ y_i\mapsto y_{3i}, \\
\lambda &: \zeta \mapsto \zeta^2,~ y_i\mapsto y_{2i}.
\end{align*}

Note that $\tau \lambda(y_i)=y_i$ for $0\le i\le 4$.

Define $z_i$ ($0\le i\le 4$) by the same way as in Case 4.1.
Then we get $k(\zeta)(x_i:0\le i\le 4)^{\langle \sigma\rangle}=k(\zeta)(z_i:0\le i\le 4)$.

Since $\tau \lambda(z_i)=z_i$, we find that $k(\zeta)(z_i:0\le
i\le 4)^{\langle \tau,\lambda\rangle}=k(\zeta)(z_i:0\le i\le
4)^{\langle \tau\lambda,\tau\rangle}=k(z_i:0\le i\le 4)^{\langle
\tau\rangle}$.

Define $u_i$ ($1\le i\le 4$) by the same way as in Case 4.1.
Then $k(z_i:0\le i\le 4)^{\langle \tau\rangle}$ is $k$-rational.
Thus $k(\zeta)(x_i: 0\le i\le 4)^{\langle \sigma,\tau,\lambda\rangle}$ is $k$-rational.

\medskip
\begin{Case}{4.3} $\pi\simeq C_2$.
We find that $\pi=\langle\lambda\rangle$ with $\lambda(\zeta)=\zeta^{-1}$. \end{Case}

The proof is similar to Case 4.2 except that $\lambda(y_i)=y_{4i}$ for $0\le i\le 4$.
In this situation, $\tau^2\lambda(y_i)=y_i$ for $0\le i\le 4$.
The details are omitted.

\bigskip
\begin{Case}{5} $G=D_5$. \end{Case}

The proof is similar to Case 4.
By Theorem \ref{t3.2} again, it remains to consider the case $\fn{char}k \ne 5$.

Write $\zeta=\zeta_5$.
Recall the automorphisms $\sigma$ and $\tau$ in Case 4.
It follows that $G=\allowbreak D_5=\langle \sigma,\tau^2\rangle$.

Using the same change of variables as in Case 4, we find that
\[
\tau^2: u_1\leftrightarrow u_3, ~ u_2\leftrightarrow u_4.
\]

In case $\zeta\in k$, $k(u_i:1\le i\le 4)^{\langle
\tau\rangle}=k(u_1,u_3)^{\langle\tau\rangle}(t_1,t_2)$ with
$\tau(t_1)=t_1$, $\tau(t_2)=t_2$ by applying Part (1) of Theorem
\ref{t2.2}. Hence $k(u_i:1\le i\le 4)^{\langle\tau\rangle}$ is
$k$-rational.

The case $\pi=\fn{Gal} (k(\zeta)/k)\simeq C_2$ is similar to Case 4.2.

Finally consider the case $\pi=\fn{Gal}(k(\zeta)/k)=\langle\lambda\rangle$ where $\lambda(\zeta)=\zeta^2$.

In this case, $\tau^2\lambda^2(y_i)=y_i$ for $0\le i\le 4$.

It remains to solve the rationality of $k(\zeta)(u_i:1\le i\le 4)^{\langle \tau^2, \lambda\rangle}$ where
\begin{align*}
\tau^2 &: \zeta \mapsto \zeta,~ u_1\leftrightarrow u_3,~ u_2\leftrightarrow u_4, \\
\lambda &: \zeta \mapsto \zeta^2,~ u_1\mapsto u_4\mapsto u_3
\mapsto u_2 \mapsto u_1.
\end{align*}

\medskip
\begin{Case}{5.1} $\fn{char} k=2$. \end{Case}

Define $v_1=u_2/u_4$, $v_2=u_1/u_3$, $v_3=u_1+u_3$, $v_4=u_2+u_4$.

Then we find that $k(u_i:1\le i\le 4)=k(v_i:1\le i\le 4)$ and
\begin{align*}
\tau^2 &: \zeta\mapsto \zeta,~ v_1 \mapsto 1/v_1,~ v_2\mapsto 1/v_2,~ v_3 \mapsto v_3,~ v_4\mapsto v_4, \\
\lambda &: \zeta \mapsto \zeta^2, ~ v_1 \mapsto v_2 \mapsto 1/v_1,~ v_3 \leftrightarrow v_4.
\end{align*}

Define
\[
w_1=\frac{v_1-\frac{1}{v_1}}{v_1v_2-\frac{1}{v_1v_2}}, \quad
w_2=\frac{v_2-\frac{1}{v_2}}{v_1v_2-\frac{1}{v_1v_2}}.
\]

By Theorem \ref{t2.8}, $k(\zeta)(v_i:1\le i\le 4)^{\langle \tau^2 \rangle}=k(\zeta)(w_1,w_2,v_3,v_4)$.

The action of $\lambda$ on $w_1$, $w_2$ is given by
\[
\lambda: w_1\mapsto
\frac{v_2-\frac{1}{v_2}}{\frac{v_2}{v_1}-\frac{v_1}{v_2}},~
w_2\mapsto
\frac{\frac{1}{v_1}-v_1}{\frac{v_2}{v_1}-\frac{v_1}{v_2}},~
w_1/w_2\mapsto -w_2/w_1.
\]

Using computer computation, it is easy to verify that
\[
\frac{\frac{1}{v_1}-v_1}{\frac{v_2}{v_1}-\frac{v_1}{v_2}}=\frac{w_1}{w_1^2-w_2^2}.
\]

Define $w_3=w_2/(w_1+w_2)$, $w_4=w_1+w_2$.
We find that $k(w_1,w_2)=k(w_3,w_4)$ and
\[
\lambda: \zeta\mapsto \zeta^2,~ w_3\mapsto w_3+1,~ w_4 \mapsto
1/w_4,~ v_3 \leftrightarrow v_4.
\]

Define $w_5=1/(1+w_4)$. Then $\lambda(w_5)=w_5 +1$.

Apply Part (1) of Theorem \ref{t2.2} to
$k(\zeta)(w_3,w_4,v_3,v_4)$ with $L=k(\zeta)$. We get
$k(\zeta)(w_3,w_5,v_3,v_4)^{\langle
\lambda\rangle}=k(\zeta)^{\langle
\lambda\rangle}(t_1,t_2,t_3,t_4)$ with $\lambda(t_i)=t_i$ for $1
\le i \le 4$. Since $k(\zeta)^{\langle
\lambda\rangle}(t_1,t_2,t_3,t_4)$ is $k$-rational, we are done.

\medskip
\begin{Case}{5.2} $\fn{char}k \ne 2$. \end{Case}

Define $v_1=u_2-u_4$, $v_2=u_1-u_3$, $v_3=u_1+u_3$, $v_4=u_2+u_4$.

We find that
\begin{align*}
\tau^2 &: \zeta\mapsto \zeta, ~ v_1\mapsto -v_1,~ v_2\mapsto -v_2,~ v_3\mapsto v_3,~ v_4\mapsto v_4, \\
\lambda &: \zeta \mapsto \zeta^2,~ v_1\mapsto v_2\mapsto -v_1,~ v_3\mapsto v_4\mapsto v_3.
\end{align*}

Apply Part (1) of Theorem \ref{t2.2}. We find $k(\zeta)(v_i:1\le
i\le 4)=k(\zeta)(v_1,v_2)(t_1,t_2)$ with
$\tau^2(t_1)=\lambda(t_1)=t_1$, $\tau^2(t_2)=\lambda(t_2)=t_2$.

Now we have $k(\zeta)(v_1,v_2)^{\langle \tau^2\rangle}=k(\zeta)(w_1,w_2)$ where $w_1=v_1/v_2$, $w_2=v_1v_2$.

Note that $\lambda(w_1)=-1/w_1$, $\lambda(w_2)=-w_2$. Apply Part
(1) of Theorem \ref{t2.2}. We may write
$k(\zeta)(w_1,w_2)^{\langle\lambda\rangle}=k(\zeta)(w_1)^{\langle\lambda\rangle}(t_3)$
with $\lambda(t_3)=t_3$.

Note that $k(\zeta)(w_1)^{\langle\lambda^2\rangle}=k(\zeta+\zeta^{-1})(w_1)$.
Note also that $k(\zeta+\zeta^{-1})=k(\sqrt{5})$ with $\lambda(\sqrt{5})=-\sqrt{5}$.

It follows that
$k(\zeta)(w_1)^{\langle\lambda\rangle}=k(\sqrt{5})(w_1)^{\langle\lambda\rangle}=k(u,v)$
where $u=\sqrt{5}(w_1+(1/w_1))$, $v=w_1-(1/w_1)$ with
$u^2-5v^2=20$.

Define $x$ and $y$ by $u=5+x$, $v=1+y$.
We get $k(u,v)=k(x,y)$ with a relation $x^2-y^2+10x-10y=0$.
Diving the relation by $x^2$, we get $1-(y/x)^2-10/x-10(y/x)(1/x)=0$.
Hence $1/x \in k(y/x)$.
Thus $x,y\in k(y/x)$.
We conclude that $k(u,v)=k(x,y)=k(y/x)$ is $k$-rational.
\end{proof}

\begin{theorem} \label{t3.5}
Let $k$ be any field, $G$ be any subgroup of $S_n$ where $n \le
5$. Then $k(G)$ is $k$-rational.
\end{theorem}

\begin{proof}
Without loss of generality, we may assume $n \le |G|$.

Let $G$ acts on the rational function field
$k(x_1,x_2,\cdots,x_n)$ by $\sigma\cdot x_i=x_{\sigma(i)}$ for any
$\sigma\in G$, any $1\le i\le n$. Recall that $k(G)=k(x(g): g \in
G)^G$ where $h \cdot x(g)=x(hg)$ for all $h,g \in G$. We may imbed
the $G$-space $\oplus_{1 \le i \le n} k \cdot x_i$ into $\oplus_{g
\in G} k \cdot x(g)$. Apply Theorem \ref{t3.3}, Theorem
\ref{t3.4}, and Part (1) of Theorem \ref{t2.2}.

\end{proof}

\section{Transitive subgroups of $S_7$ and $S_{11}$}
\setcounter{equation}{0}
\begin{theorem} \label{t4.7}
Let $G=G_1\times G_2$ be a finite group, $K$ be a field on which $G$
acts such that (i) $[K:k]<\infty$ where $k=K^G$, and (ii) the kernel
of the induced morphism $G\rightarrow\fn{Gal}(K/k)$ is $G_1$. Let
$K(x(g):g\in G)$ be the rational function field with $G$-actions
such that $G$ acts on $K$ as before and $h\cdot x(g)=x(hg)$ for any
$h, g\in G$. If $k(G_1)$ is $k$-rational, then $K(x(g):g\in G)^G$ is
also $k$-rational. Conversely, if $K(x(g):g\in G)^G$ is
$k$-rational, then $k(G_1)$ is stably $k$-rational.
\end{theorem}

\begin{proof}
By assumptions, $G_2$ acts on $K$ faithfully. Apply Part (1) of
Theorem \ref{t2.2} to the subfield $K(x(h):h\in G_2)$. There is a
matrix $T\in GL_m(K)$ where $m=|G_1|$ and define $u_1,\ldots,u_m$ by
$$
\begin{pmatrix} u_1 \\ u_2 \\ \vdots \\ u_m \end{pmatrix}
= T \cdot \begin{pmatrix} x(h_1) \\ x(h_2) \\ \vdots \\ x(h_m)
\end{pmatrix}
$$
where $G_2=\{h_1,h_2,\ldots,h_m\}$. We obtain $h(u_i)=u_i$ for all
$h\in G_2$, for any $1\leq i\leq m$.

For any $\lambda\in G_1$, define $u^{(\lambda)}_i=\lambda(u_i)$ for
$1\leq i\leq m, \lambda\in G_1$. It is easy to verify that
$h(u^{(\lambda)}_i)=u^{(\lambda)}_i$ for any $h\in G_2$, any
$\lambda\in G_1$, any $1\leq i\leq m$; moreover, $\sum_{1\leq i\leq
m}K\cdot u^{(\lambda)}_i=\sum_{1\leq i\leq m}K\cdot x(\lambda h_i)$.
Hence $K(x(g):g\in G)=K(u^{(\lambda)}_i:\lambda\in G_1, 1\leq i\leq
m)$.

We find that $K(x(g):g\in G)^{G_2}=K(u^{(\lambda)}_i:\lambda\in G_1,
1\leq i\leq m)^{G_2}=k(u^{(\lambda)}_i:\lambda\in G_1, 1\leq i\leq
m)$. The action of $G_1$ is given by $\tau\cdot
u^{(\lambda)}_i=u^{(\tau\lambda)}_i$ for all $\tau,\lambda\in G_1$,
all $1\leq i\leq m$.

Apply Part (1) of Theorem \ref{t2.2} to
$k(u^{(\lambda)}_i:\lambda\in G_1, 1\leq i\leq m)$ with
$L=k(u^{(\lambda)}_1:\lambda\in G_1)$. We find
$t_1,t_2,\ldots,t_{m(n-1)}$ (where $n=|G_1|$) such that
$k(u^{(\lambda)}_i:\lambda\in G_1, 1\leq i\leq
m)=k(u^{(\lambda)}_1:\lambda\in G_1)(t_1,t_2,\ldots,t_{m(n-1)})$
satisfying that $\tau(t_j)=t_j$ for any $\tau\in G_1$, any $1\leq
j\leq m(n-1)$.

If $k(G_1)$ is $k$-rational, then $k(u^{(\lambda)}_i:\lambda\in G_1,
1\leq i\leq m)^{G_1}=k(u^{(\lambda)}_1:\lambda\in
G_1)^{G_1}(t_1,t_2,\ldots,t_{m(n-1)})=k(G_1)(t_1,t_2,\ldots,t_{m(n-1)})$
is $k$-rational.

On the other hand, if $K(x(g):g\in G)^G$ is $k$-rational, then
$k(u^{(\lambda)}_i:\lambda\in G_1, 1\leq i\leq
m)^{G_1}=k(G_1)(t_1,t_2,\ldots,t_{m(n-1)})$ is $k$-rational. Then
$k(G_1)$ is stably $k$-rational.
\end{proof}

\begin{theorem} \label{t4.4}
Let $n=de$ with $\gcd\{d,e\}=1$. Let $\sigma$ be an automorphism of
the rational function field $K(x_i:0\leq i\leq n-1)$ defined by
$$
\sigma: x_0\mapsto x_1\mapsto \cdots \mapsto x_{n-1}\mapsto x_0
$$
such that $\sigma(K)=K$ and $[K:k]=d$ where $k=K^{\langle
\sigma\rangle}$. Then $K(x_i:0\leq i\leq n-1)^{\langle
\sigma\rangle}$ is $k$-rational if and only if $k(C_e)$ is
$k$-rational.
\end{theorem}
\begin{proof}
Apply Theorem \ref{t4.7}. We find that
$\langle\sigma\rangle=\langle\sigma^d\rangle\times\langle\sigma^e\rangle$
and $\sigma^d$ acts trivially on $K$. Note that $k(C_e)$ is
$k$-rational if and only if $k(C_e)$ is stably $k$-rational by
\cite[p. 319, Remark 5.7]{Le}.
\end{proof}

\begin{theorem}\label{t4.1}
Let $k$ be any field, $G_{42}=\langle\sigma, \tau\rangle$ be the
group in Definition 3.1 acting on the rational function field
$k(x_i:0\leq i\leq 6)$ via $k$-automorphisms defined by
$\sigma:x_i\mapsto x_{i+1}, \tau:x_i\mapsto x_{3i}$ for $0\leq i\leq
6$. Then $k(x_i:0\leq i\leq 6)^{G_{42}}$ is $k$-rational.
\end{theorem}
\begin{proof}
Because of Theorem \ref{t3.2}, it remains to consider the case
$\fn{char}k\neq 7$.

Write $G=G_{42}, \zeta=\zeta_7$ where $\zeta_7$ is a primitive
$7$th-root of unity. Define
$\pi=\fn{Gal}(k(\zeta)/k)=\langle\lambda\rangle$. Then $\pi\simeq
C_6,C_3,C_2$ or $\{1\}$.

\medskip
\begin{Case}{1} $\pi=\{1\}$, i.e. $\zeta\in k$. \end{Case}

For $0\leq i\leq 6$, define
\begin{equation}
y_i=\sum_{0\le j\le 6}
\zeta^{-ij} x_j. \label{eq4.1}
\end{equation}

Then $\sigma(y_i)=\zeta^i y_i$, $\tau(y_i)=y_{5i}$ for $0\le i\le
6$.

It follows that $k(x_i:0\le i\le 6)^{\langle
\sigma\rangle}=k(y_i:0\le i\le 6)^{\langle\sigma\rangle}=k(z_i:0\le
i\le 6)$ where $z_0=y_0$, $z_1=y_1^7$, $z_i=y_i/y_1^i$ for $2\le
i\le 6$.

Note that
\begin{align*}
\tau:&z_0\mapsto z_0,~ z_1\mapsto z_1^5 z_5^7, ~ z_2\mapsto
z_3/(z_1z_5^2),~ z_3\mapsto 1/(z_1^2z_5^3),~ z_4\mapsto z_6/(z_1^2
z_5^4), \\
&z_5\mapsto z_4/(z_1^3 z_5^5),z_6\mapsto z_2/(z_1^4 z_5^6).
\end{align*}

Define $u_1=z_3/z_4, u_2=z_5/z_6, u_3=z_1z_4z_5/z_2,
u_4=z_1z_4z_6/z_3, u_5=z_1z_2z_6, u_6=z_2z_3/z_5$.

It follows that $k(z_i:1\le i\le 6)=k(u_i:1\le i\le 6)$ and
$$\tau:u_1\mapsto u_2\mapsto u_3\mapsto u_4\mapsto u_5\mapsto u_6\mapsto
u_1.$$

By Theorem \ref{t2.7}, $k(u_i:1\le i\le 6)^{\langle\tau\rangle}$ is
$k$-rational. Hence the result.

\medskip
\begin{Case}{2} $\pi\simeq C_6$. We may assume that $\pi=\langle\lambda\rangle$ write
$\lambda(\zeta)=\zeta^3$.\end{Case}

Extend the action of $G$ on $k(x_i:0\le i\le 6)$ to the action of
$\langle G,\lambda\rangle$ on $k(\zeta)(x_i:0\le i\le 6)$ as in the
proof Case 4.2 of Theorem \ref{t3.4}. Define $y_i$ by the same
formula as (\ref{eq4.1}). We get
\begin{align*}
\sigma &: \zeta\mapsto \zeta,~ y_i\mapsto \zeta^i y_i, \\
\tau &: \zeta\mapsto \zeta,~ y_i\mapsto y_{5i}, \\
\lambda &: \zeta \mapsto \zeta^3,~ y_i\mapsto y_{3i}.
\end{align*}

Note that $\tau\lambda(y_i)=y_i$ for $0\le i\le6$.

The proof is almost the same as that in Case 4.2 of Theorem
\ref{t3.4}. Define $z_i, u_j$ by the same way as in Case 1. It
follows that $k(\zeta)(x_i:0\le i\le
6)^{\langle\sigma,\tau,\lambda\rangle}=k(\zeta)(z_i:0\le i\le
6)^{\langle\tau,\lambda\rangle}=k(\zeta)(u_i:1\le i\le
6)^{\langle\tau\lambda,\tau\rangle}(z_0)=k(u_i:1\le i\le
6)^{\langle\tau\rangle}(z_0)$.

Since $k(u_i:1\le i\le 6)^{\langle\tau\rangle}$ is $k$-rational by
Theorem \ref{t2.7}, it follows that $k(x_i:0\le i\le
6)^{\langle\sigma,\tau\rangle}=k(\zeta)(x_i:0\le i\le
6)^{\langle\sigma,\tau,\lambda\rangle}$ is $k$-rational.

\begin{Case}{3} $\pi\simeq C_3$ or $C_2$. \end{Case}

The proof is almost the same and is omitted.
\end{proof}
\begin{theorem}\label{t4.2}
Let $k$ be any field, $D_7$ be the group $G_{14}$ in Definition 3.1.
Let $D_7=\langle\sigma,\tau^3\rangle$ act on the rational function
field $k(x_i:0\le i\le 6)$ by $\sigma:x_i\mapsto x_{i+1},
\tau^3:x_i\mapsto x_{-i}$ for $0\le i\le 6$. Then $k(x_i:0\le i\le
6)^{D_7}$ is $k$-rational.
\end{theorem}
\begin{proof}
The proof is similar to that of Theorem \ref{t4.1}. Compare the
proof of Case 5 in Theorem \ref{t3.4}.

Again we will consider the case $\fn{char}k\neq 7$ only.

Write $G=D_7=\langle\sigma,\tau^3\rangle$ and $\zeta=\zeta_7,
\pi=\fn{Gal}(k(\zeta)/k)=\langle\lambda\rangle$.

Define $u_j$ by the same way as in the proof of Theorem \ref{t4.1}.
It remains to show that $k(\zeta)(u_i:1\le i\le
6)^{\langle\tau^3,\lambda\rangle}$ is $k$-rational.

\medskip
\begin{Case}{1} $\pi=\{1\}$, i.e. $\zeta\in k$ and $\lambda=1$. \end{Case}

We find that
$$ \tau^3:u_1\leftrightarrow u_4, u_2\leftrightarrow u_5, u_3\leftrightarrow
u_6. $$

Apply Part (1) of Theorem \ref{t2.2}. We find that $k(u_i:1\le i\le
6)^{\langle\tau^3\rangle}=k(u_1,u_4)^{\langle\tau^3\rangle}(t_1,t_2,t_3,t_4)$
with $\tau^3(t_i)=t_i$ for $1\le i\le 4$. Since
$k(u_1,u_4)^{\langle\tau^3\rangle}=k(u_1+u_4,u_1u_4)$ is
$k$-rational, we are done.

\medskip
\begin{Case}{2} $\pi\simeq C_3$, i.e. we may assume that $\lambda(\zeta)=\zeta^2$. \end{Case}

We find that
\begin{align*}
\tau^3 &: \zeta\mapsto \zeta, ~ u_1\leftrightarrow u_4,~
u_2\leftrightarrow u_5,~ u_3\leftrightarrow
u_6, \\
\lambda &: \zeta \mapsto \zeta^2,~ u_1\mapsto u_5\mapsto
u_3\mapsto u_1,~ u_2\mapsto u_6\mapsto u_4\mapsto u_2.
\end{align*}

Note that $\langle\tau^3\lambda\rangle=\langle\tau^3,\lambda\rangle$
and
$$
\tau^3\lambda: \zeta \mapsto \zeta^2,~ u_1\mapsto u_2\mapsto
u_3\mapsto u_4 \mapsto u_5\mapsto u_6\mapsto u_1.
$$

Apply Theorem \ref{t4.4} and Theorem \ref{t2.7}. We find that
$k(\zeta)(u_i:1\le i\le 6)^{\langle\tau^3\lambda\rangle}$ is
$k$-rational.

\medskip
\begin{Case}{3} $\pi\simeq C_2$, i.e. $\lambda(\zeta)=\zeta^{-1}$.
\end{Case}

We find that
\begin{align*}
\tau^3 &: \zeta\mapsto \zeta, ~ u_1\leftrightarrow u_4,~
u_2\leftrightarrow u_5,~ u_3\leftrightarrow
u_6, \\
\lambda &: \zeta \mapsto \zeta^{-1}, ~ u_1\leftrightarrow u_4,~
u_2\leftrightarrow u_5,~ u_3\leftrightarrow u_6.
\end{align*}

Consider $k(\zeta)(u_i:1\le i\le 6)^{\langle\tau^3,
\lambda\rangle}=\{ k(\zeta)(u_i:1\le i\le
6)^{\langle\tau^3\lambda\rangle}\}^{\langle\tau^3\rangle}.$ Done.

\medskip
\begin{Case}{4} $\pi\simeq C_6$, i.e. $\lambda(\zeta)=\zeta^{3}$.\end{Case}

We find that
\begin{align*}
\tau^3 &: \zeta\mapsto \zeta, ~ u_1\leftrightarrow u_4,~
u_2\leftrightarrow u_5,~ u_3\leftrightarrow
u_6, \\
\lambda &: \zeta \mapsto \zeta^{3}, ~ u_1\mapsto u_6\mapsto
u_5\mapsto u_4 \mapsto u_3\mapsto u_2\mapsto u_1.
\end{align*}

Note that $\tau^3\lambda^3(\zeta)=\zeta^{-1}$ and
$\tau^3\lambda^3(u_i)=u_i$ for $1\le i\le 6$. We have

\begin{align*}
&k(\zeta)(u_i:1\le i\le 6)^{\langle\tau^3,
\lambda\rangle}=k(\zeta)(u_i:1\le i\le 6)^{\langle\tau^3\lambda^3,
\lambda\rangle}\\
&=\{k(\zeta)(u_i:1\le i\le
6)^{\langle\tau^3\lambda^3\rangle}\}^{\langle\lambda\rangle}=k(\eta)(u_i:1\le
i\le 6)^{\langle\lambda\rangle}
\end{align*}
where $\eta=\zeta+\zeta^{-1}$. Apply Theorem \ref{t4.4} and Theorem
\ref{t2.7}.
\end{proof}
\begin{theorem}\label{t4.3}
Let $k$ be any field, $G_{21}$ be the group in Definition 3.1. Let
$G_{21}=\langle\sigma,\tau^2\rangle$ act on the rational function
field $k(x_i:0\le i\le 6)$ by $\sigma:x_i\mapsto x_{i+1},
\tau^2:x_i\mapsto x_{2i}$ for $0\le i\le 6$. Then $k(x_i:0\le i\le
6)^{G_{21}}$ is $k$-rational.
\end{theorem}
\begin{proof}
The proof is similar to those of Theorem \ref{t4.1} and Theorem
\ref{t4.2}.

We may assume that $\fn{char}k\neq 7$. Write
$G=G_{21}=\langle\sigma,\tau^2\rangle, \zeta=\zeta_7,
\pi=\fn{Gal}(k(\zeta)/k)=\langle\lambda\rangle$. It remains to show
that $k(\zeta)(u_i:1\le i\le 6)^{\langle\tau^2,\lambda\rangle}$ is
$k$-rational where $\tau^2: u_1\mapsto u_3\mapsto u_5\mapsto u_1,~
u_2\mapsto u_4\mapsto u_6\mapsto u_2$.

\medskip
\begin{Case}{1} $\pi\simeq\{1\}$ or $C_3$. \end{Case}

This is similar to Case 1 and Case 3 of Theorem \ref{t4.2}. The
proof is omitted.

\medskip
\begin{Case}{2} $\pi\simeq C_2$, i.e. we may assume that $\lambda(\zeta)=\zeta^{-1}$. \end{Case}

We find that
\begin{align*}
\tau^2 &: \zeta\mapsto \zeta, ~ u_1\mapsto u_3\mapsto u_5\mapsto
u_1,~ u_2\mapsto u_4\mapsto u_6\mapsto u_2, \\
\lambda &: \zeta \mapsto \zeta^{-1}, ~u_1\leftrightarrow u_4,~
u_2\leftrightarrow u_5,~ u_3\leftrightarrow u_6 .
\end{align*}

Since $\langle\tau^2\lambda\rangle=\langle\tau^2,\lambda\rangle$ and
$$
\tau^2\lambda: \zeta\mapsto \zeta^{-1}, ~ u_1\mapsto u_6\mapsto
u_5\mapsto u_4\mapsto u_3\mapsto u_2\mapsto u_1,
$$
we may apply Theorem \ref{t4.4} and Theorem \ref{t2.7}. Done.

\medskip
\begin{Case}{3} $\pi\simeq C_6$, i.e. $\lambda(\zeta)=\zeta^{3}$. \end{Case}

We find that
\begin{align*}
\tau^2 &: \zeta\mapsto \zeta, ~ u_1\mapsto u_3\mapsto u_5\mapsto
u_1,~ u_2\mapsto u_4\mapsto u_6\mapsto u_2, \\
\lambda &: \zeta \mapsto \zeta^{3}, ~ u_1\mapsto u_6\mapsto
u_5\mapsto u_4\mapsto u_3\mapsto u_2\mapsto u_1 .
\end{align*}

Thus $k(\zeta)(u_i:1\le i\le 6)^{\langle\tau^2,
\lambda\rangle}=\{k(\zeta)(u_i:1\le i\le
6)^{\langle\tau^2\lambda^2\rangle}\}^{\langle\lambda\rangle}$\\
$=\{k(\zeta)^{\langle\lambda^{2}\rangle}(u_i:1\le i\le
6)\}^{\langle\lambda\rangle}$. Apply Theorem \ref{t4.4} and Theorem
\ref{t2.7}.
\end{proof}

\begin{defn}\label{d4.8}
Note that $PSL_2(\bm{F}_7) \simeq GL_3(\bm{F}_2)$ is the unique
simple group of order $168$. Moreover, $GL_3(\bm{F}_2)\simeq
PGL_3(\bm{F}_2)$ is the automorphism group of the projective plane
over $\bm{F}_2$, which consists of 7 points. Thus
$PSL_2(\bm{F}_7)$ may be presented as a permutation group of
degree $7$. Define $G_{168}= \langle\sigma, \tau \rangle \subset
S_7$ by $\sigma =(1,2,3,4,5,6,7)$ and $\tau = (2,3)(4,7)$. It is
not difficult to show that $PSL_2(\bm{F}_7) \simeq G_{168}$.

\end{defn}

\begin{theorem}\label{t4.5}
Let $k$ be any field such that $\fn{char}k=0$ and $\sqrt{-7}\in
k$. Let $G_{168}=\langle\sigma,\tau\rangle$ be the group in
Definition \ref{d4.8}, which acts on the rational function field
$k(x_i:1\le i\le 7)$ by $\sigma:x_1\mapsto x_{2}\mapsto
x_{3}\mapsto x_{4}\mapsto x_{5}\mapsto x_{6}\mapsto x_{7}\mapsto
x_{1},$ $\tau:x_2\leftrightarrow x_3, x_4\leftrightarrow x_7,
x_1\mapsto x_1, x_5\mapsto x_5, x_6\mapsto x_6$. Then $k(x_i:1\le
i\le 7)^{\langle\sigma,\tau\rangle}$ is $k$-rational.
\end{theorem}
\begin{proof}
Write $G=G_{168}$.

We will apply similar techniques as in the proof of Theorem
\ref{t3.2} to solve the rational problem of the present situation.

Define $x_0=\sum_{1\le i\le 7}x_i, y_i=x_i-(x_0/7)$ for $1\le i\le
7$. Then $k(x_i:1\le i\le 7)^{\langle\sigma,\tau\rangle}=k(y_i:1\le
i\le 7)^{\langle\sigma,\tau\rangle}(x_0)$ with $\sum_{1\le i\le
7}y_i=0$. Note that $\sigma$ and $\tau$ act linearly on $\sum_{1\le
i\le 7}k\cdot y_i=\bigoplus_{1\le i\le 6}k\cdot y_i$. Apply Theorem
\ref{t2.3}. We get $k(y_i:1\le i\le
6)^{\langle\sigma,\tau\rangle}=k(y_i/y_6:1\le i\le
5)^{\langle\sigma,\tau\rangle}(t)$ with $\sigma(t)=\tau(t)=t$. In
conclusion, $k(x_i:1\le i\le
7)^{\langle\sigma,\tau\rangle}=k(y_i/y_6:1\le i\le
5)^{\langle\sigma,\tau\rangle}(t, x_0)$ with $\sigma(t)=\tau(t)=t,
\sigma(x_0)=\tau(x_0)=x_0$.

Since $\sqrt{-7}\in k$, choose the faithful representation
$G\rightarrow GL(V)$ with $\dim_kV=3$ in Theorem \ref{t2.9}. Let
$z_1, z_2, z_3$ be a basis of $V^{*}$,  the dual space of $V$. Then
$k(V)=k(z_1, z_2, z_3)$. Thus $G$ acts faithfully on $k(z_1/z_3,
z_2/ z_3)$ because $G$ is a simple group.

Apply Part (2) of theorem \ref{t2.2} to $k(y_i/y_6: 1\le i\le
5)(z_1/z_3, z_2/ z_3)$. We have that $k(y_i/y_6: 1\le i\le
5)(z_1/z_3, z_2/ z_3)^G=k(y_i/y_6: 1\le i\le 5)^G(s_1, s_2)$ with
$\sigma(s_i)=\tau(s_i)=s_i$ for $1\le i\le 2$. Thus $k(y_i/y_6: 1\le
i\le 5)^G(s_1, s_2)\simeq k(y_i/y_6: 1\le i\le 5)^G(t, x_0)=k(x_i:
1\le i\le 7)^G$.

On the other hand, apply Part (2) of Theorem \ref{t2.2} to
$k(y_i/y_6: 1\le i\le 5)(z_1/z_3, z_2/ z_3)^G$ with $L=k(z_1/z_3,
z_2/ z_3)$. We get $k(y_i/y_6: 1\le i\le 5)(z_1/z_3, z_2/
z_3)^G=k(z_1/z_3, z_2/ z_3)^G(v_j: 1\le j\le 5)$ with
$\sigma(v_j)=\tau(v_j)=v_j$ for $1\le j\le 5$.

By Theorem \ref{t2.3}, $k(V)^G=k(z_1, z_2, z_3)^G=k(z_1/z_3, z_2/
z_3)^G(w)$ for some $w$ with $\sigma(w)=\tau(w)=w$.

Hence $k(x_i:1\le i\le 7)^G\simeq k(y_i/y_6: 1\le i\le 5)(z_1/z_3,
z_2/ z_3)^G\simeq k(z_1/z_3, z_2/ z_3)^G(v_j: 1\le j\le 5)\simeq
k(z_1/z_3, z_2/ z_3)^G(w)(v_j: 2\le j\le 5)=k(V)^G(v_j: 2\le j\le
5)$. Since $k(V)^G$ is $k$-rational by Theorem \ref{t2.9}, we find
that $k(x_i: 1\le i\le 7)^G$ is $k$-rational.
\end{proof}

\bigskip
Proof of Theorem \ref{t1.4} ---------------

The transitive subgroups of $S_7$, up to conjugation, are
$$C_7, D_7, G_{21}, G_{42}, G_{168}, A_7\hbox{ and } S_7$$
where $G_{168}$ is the group in Definition \ref{d4.8}, and
$G_{42}, G_{21}$ are the group $G_{pd}$ in Definition 3.1 with
$p=7$ \cite[p. 60]{DM}.

If $G=C_7$, the rationality of $k(x_i: 1\le i\le 7)^G$ follows
from Theorem \ref{t2.7}. If $G=D_7, G_{21}, G_{42}$, the
rationality problem follows from Theorem \ref{t4.1}, Theorem
\ref{t4.2} and Theorem \ref{t4.3}. When $G=S_7$, the rationality
problem is easy.

When $G=G_{168}$ and $\bm{Q}(\sqrt{-7})\subseteq k$, the rationality
problem follows form Theorem \ref{t4.5}.\hfill$\blacksquare$

\bigskip
Proof of Theorem \ref{t1.5} ---------------

The transitive solvable subgroups of $S_{11}$ are conjugate to
subgroups of $G_{11\cdot10}=\langle\sigma,\tau\rangle$ be the
group $G_{p(p-1)}$ in Definition 3.1 with $p=11$ \cite[p.117,
Proposition 11.6; DM, p.91, Exercise 3.5.1]{Co}..

Let $G_{11\cdot10}$ act on the rational function field $k(x_i:0\le
i\le 10)$ by $\sigma:x_i\mapsto x_{i+1}, \tau:x_i\mapsto x_{2i}$
for $0\le i\le 10$. For any divisor $d$ of $10$, write $10=de$.
Define $G=\langle\sigma,\tau^e\rangle$. We will prove that the
fixed field $k(x_i:0\le i\le 10)^G$ is $k$-rational.

The case $\fn{char}k=11$ follows from Theorem \ref{t3.2}.

From now on, we may assume that $\fn{char}k\neq11$. Write
$\zeta=\zeta_{11}$ where $\zeta_{11}$ is a primitive $11$th root
of unity. Define $\pi=\fn{Gal}(k(\zeta)/k)=\langle\lambda\rangle$.
Then $\pi\simeq\{1\}, C_2, C_5$ or $C_{10}$.

The proof is similar to those of Theorem \ref{t4.2}, Theorem
\ref{t4.3}. We will indicate only the key ideas here.

For $0\le i\le 10$, define
$$y_i=\sum_{0\le i\le 10}\zeta^{-ij}x_i.$$

It follows that $\sigma(y_i)=\zeta^iy_i, \tau(y_i)=y_{6i},
\tau^e(y_i)=y_{6^ei}$.

We get $k(x_i:0\le i\le 10)^G=\{k(\zeta)(x_i:0\le i\le
10)\}^{\langle\sigma, \tau^e, \lambda\rangle}=\{k(\zeta)(y_i:0\le
i\le 10)\}^{\langle\sigma, \tau^e,
\lambda\rangle}=\{k(\zeta)(z_i:0\le i\le 10)\}^{\langle\tau^e,
\lambda\rangle}$ where $z_0=y_0, z_1=y_1^{11}, z_i=y_i/y^i_1$ for
$2\le i\le 10$.

It is not difficult to verify that
\begin{align*}
\tau:&z_0\mapsto z_0,~ z_1\mapsto z_1^6 z_6^{11}, ~ z_2\mapsto
1/(z_1z_6^2),~ z_3\mapsto z_7/(z_1z_6^3),~ z_4\mapsto z_2/(z_1^2
z_6^4), \\
&z_5\mapsto z_8/(z_1^2 z_6^5), z_6\mapsto z_3/(z_1^3 z_6^6),
z_7\mapsto z_9/(z_1^3 z_6^7), z_8\mapsto z_4/(z_1^4 z_6^8),\\
&z_9\mapsto z_{10}/(z_1^4 z_6^9),z_{10}\mapsto z_5/(z_1^5 z_6^{10}).
\end{align*}

Define $u_1=z_2/z_3, u_i=\tau^{i-1}(u_1)$ for $2\le i\le 10$.
Explicitly, the exponents of $u_j$ in terms of $z_i$ (where $1 \le
i,j \le 10$) is represented as the following matrix

$$A=\left(
    \begin{array}{rrrrrrrrrr}
0& 0 & 0 & 0 & 1 & 1 & 1 & 1 & \ \ 1 & 0 \\
1& 0 & 0 & 0 & 0 & 0 & 0 &-1 & 0 & 1 \\
-1& 0 & 1 & 1 & 0 & 0 & 0 & 0 & 0 & 0 \\
0& 0 & 0 & 0 & 0 & 0 &-1 & 0 & 1 & 1 \\
0&  0 & 0 & 0 &-1 & 0 & 1 & 1 & 0 & 0 \\
0&  1 & 1 & 0 & 0 & 0 & 0 & 0 & 0 &-1 \\
0& -1 & 0 & 1 & 1 & 0 & 0 & 0 & 0 & 0 \\
0&  0 & 0 & 0 & 0 &-1 & 0 & 1 & 1 & 0 \\
0&  0 &-1 & 0 & 1 & 1 & 0 & 0 & 0 & 0 \\
0&  0 & 0 &-1 & 0 & 1 & 1 & 0 & 0 & 0\\
    \end{array}
  \right).
$$
For example, the first column of $A$ denotes $u_1=z_2/z_3$.

Since det$A=1$, we find that $k(z_i:1\le i\le 10)=k(u_i:1\le i\le
10)$. The action of $\tau$ is given by
\begin{equation}\label{4.2}
    \tau: u_1\mapsto u_2\mapsto\cdots\mapsto u_{10}\mapsto u_1.
\end{equation}

\medskip
\begin{Case}{1} $\pi\simeq C_{10}$. We may assume that $\lambda(\zeta)=\zeta^{2}$. \end{Case}

It follows that
$$\lambda: \zeta\mapsto \zeta^2, y_i\mapsto y_{2i}.$$

Thus $\tau\lambda(y_i)=y_i$ for $0\le i\le 10$. Hence we find that
$$
\lambda: \zeta\mapsto \zeta^2, u_1\mapsto u_{10}\mapsto u_9\mapsto
u_8\mapsto\cdots\mapsto u_{2}\mapsto u_1.
$$

Now we have $k(\zeta)(u_i:1\le i\le 10)^{\langle\tau^e,
\lambda\rangle}=k(\zeta)(u_i:1\le i\le 10)^{\langle\tau^e\lambda^e,
\lambda\rangle}=\{k(\zeta)^{\langle\lambda^e\rangle}(u_i:1\le i\le
10)\}^{\langle\lambda\rangle}$. Apply Theorem \ref{t4.4} and Theorem
\ref{t2.7}.

\medskip
\begin{Case}{2} $\pi\simeq C_{5}$, i.e. We may assume that $\lambda(\zeta)=\zeta^{4}$. It follows that
$\lambda(y_i)=y_{4i}$.\end{Case}

We will consider $k(\zeta)(u_i:1\le i\le 10)^{\langle\tau^e,
\lambda\rangle}$.

\medskip
\begin{Case}{2.1} $e=1$.\end{Case}

We find that $k(\zeta)(u_i:1\le i\le 10)^{\langle\tau,
\lambda\rangle}= \{k(\zeta)(u_i:1\le i\le 10)
^{\langle\tau^2\lambda\rangle}\}^{\langle\tau\rangle}=k(u_i:1\le
i\le 10)\}^{\langle\tau\rangle}$ is $k$-rational by Theorem
\ref{t2.7}.

\medskip
\begin{Case}{2.2} $e=5, 2, 10$.\end{Case}

Consider the case $e=5$ first.

Since $(\tau^5)^2=1$ and $[k(\zeta):k]=5$, we find that
$\langle\tau^5\lambda\rangle=\langle\tau^5,\lambda\rangle$. We may
apply Theorem \ref{t4.4} and Theorem \ref{t2.7}.

The remaining situations are similar.

\medskip
\begin{Case}{3} The proof when $\pi\simeq C_{2}$ is similar to the case when $\pi\simeq
C_5$. Thus the proof is omitted.\end{Case}

\medskip
\begin{Case}{4} $\pi\simeq \{1\}$, i.e. $\zeta\in k$. We may apply
Theorem \ref{t2.7} directly (with the aid of Theorem \ref{t2.2}
when necessary).\end{Case}\hfill$\blacksquare$

\begin{remark}
The above theorem may be generalized to the case of other prime
numbers  $p$, provided that we can change  the variables
$\{z_i:1\le i\le p-1\}$ to the variables $\{u_i:1\le i\le p-1\}$
such that the formula (\ref{4.2}) is valid, i.e. $\tau:u_1\mapsto
u_2\mapsto\cdots\mapsto u_{p-1}\mapsto u_1$. This is the case for
prime numbers $p \le 41$ by Samson Breuer \cite{Br}.

However, this condition is not met for all prime numbers. For
example, if $p=47$ and the above condition is satisfied, then
$\bm{Q}(x_i:0\le i\le 46)^{\langle\sigma\rangle}$ is
$\bm{Q}$-rational, which is impossible because of Swan's
counter-example \cite{Le,Sw}.
\end{remark}
\newpage
\renewcommand{\refname}{\centering{References}}


\begin{thebibliography}{AHK}

\bibitem[AHK]{AHK}
H. Ahmad, M. Hajja and M. Kang, \textit{Rationality of some projective linear actions},
J. Algebra 228 (2000), 643--658.

\bibitem[Br]{Br}
S. Breuer, \textit{Zyklische Minimalbasis zusammengesetzten
Grades}, J. reine angew. Math. 166 (1932), 54--58.

\bibitem[Bu]{Bu}
W. Burnside,
\textit{The alternating functions of three and of four variables},
Messenger of Math. 37 (1907), 165--166.

\bibitem[CHK]{CHK}
H. Chu, S.-J. Hu and M. Kang,
\textit{A rationality problem of Certain $A_4$ actions},
in ``Recent Developments in Algebra and Related Areas" edited by C. Dong and F. Li,
Higher Education Press and International Press, Beijing-Boston, 2009.

\bibitem[Co]{Co}
P. M. Cohn, Algebra vol.2, 2nd edition, John Wiley \& Sons, New York, 1982.

\bibitem[DM]{DM}
J.D. Dixon and B. Mortimer, Permutation groups, GTM vol. 163,
Springer-Verlag, New York, 1996.

\bibitem[GMS]{GMS}
S. Garibaldi, A. Merkurjev and J. P. Serre,
\textit{Cohomological invariants in Galois cohomology},
AMS Univ. Lecture Series, vol. 28,
Amer. Math. Soc., Providence, RI, 2003.

\bibitem[HK1]{HK1}
M. Hajja and M. Kang,
\textit{Finite group actions on rational function fields}, J. Algebra 149 (1992), 139--154.

\bibitem[HK2]{HK2}
M. Hajja and M. Kang,
\textit{Three-dimensional purely monomial group actions},
J. Algebra 170 (1994), 850--860.

\bibitem[HK3]{HK3}
M. Hajja and M. Kang, \textit{Some actions of symmetric groups},
J. Algebra 177 (1995), 511--535.

\bibitem[HKY]{HKY}
A. Hoshi, H. Kitayama and A. Yamasaki, \textit{Rationality problem
of three-dimensional monomial group actions}, J. Algebra 341
(2011), 45--108.


\bibitem[HR]{HR}
A. Hoshi and Y. Rikuna,
\textit{Rationality problem of three-dimensional purely monomial group actions: the last case},
Math. Comp. 77 (2008), 1823--1829.

\bibitem[Ka]{Ka}
M. Kang, \textit{Rationality problem of $GL_4$ group actions},
Advances in Math. 181 (2004), 321--352.

\bibitem[Ke]{Ke}
G. Kemper, \textit{A constructive approach to Noether's problem},
Manuscripta math. 90 (1996), 343--363.

\bibitem[KP]{KP}
M. Kang and B. Plans, \textit{Reduction theorems for Noether's
problem}, Proc. Amer. Math. Soc. 137 (2009), 1867--1874.

\bibitem[KY]{KY}
H. Kitayama and A. Yamasaki,
\textit{The rationality problem for four-dimensional linear actions},
J. Math. Kyoto Univ. 49 (2009), 359--380.

\bibitem[KZ]{KZ}
M. Kang and J. Zhou, \textit{The rationality problem for finite subgroups of $GL_4(\bm{Q})$},
J. Algebra 368 (2012), 53--69.

\bibitem[Le]{Le}
H. W. Lenstra Jr.,
\textit{Rational functions invariant under a finite abelian group},
Invent. Math. 25 (1974), 299--325.

\bibitem[Ma]{Ma}
T. Maeda, \textit{Noether's problem for $A_5$},
J. Algebra 125 (1989), 418--430.

\bibitem[No]{No}
E. Noether,
\textit{Rationale Funktionenk\"orper}, Jber. Deutsch. Math.-Verein. 22 (1913), 316--319.

\bibitem[Sa]{Sa}
D. J. Saltman, \textit{Generic Galois extensions and problems in
field theory}, Advances in Math. 43 (1982), 250--283.

\bibitem[Sw]{Sw}
R. G. Swan, \textit{Noether's problem in Galois theory}, in ``Emmy
Noether in Bryn Mawr", edited by B. Srinivasan and J. Sally,
Springer-Verlag, Berlin, 1983, pp.~21--40.

\bibitem[TW]{TW}
J. A. Tyrrell and C. M. Williams,
\textit{On the pure transcendence of certain fields of rational functions},
Bull. London Math. Soc. 1 (1969), 75--78.

\bibitem[Ya]{Ya}
A. Yamasaki, \textit{Negative solutions to three-dimensional
monomial Noether problem}, J. Algebra 370 (2012), 46--78.

\end{thebibliography}
\end{document}